\documentclass[a4paper,11pt]{amsart}
\usepackage{geometry}
\usepackage{fullpage}
\usepackage[parfill]{parskip}   
\usepackage{graphicx}
\usepackage{amssymb}
\usepackage{amsmath}
\usepackage{epstopdf}
\usepackage{amsmath}

\usepackage{amsmath,amsthm,amscd,amssymb}
\usepackage{latexsym}
\usepackage[colorlinks,citecolor=red,pagebackref,hypertexnames=false]{hyperref}
 
\numberwithin{equation}{section}

\theoremstyle{plain}
\newtheorem{theorem}{Theorem}[section]

\newtheorem{lemma}[theorem]{Lemma}
\newtheorem{corollary}[theorem]{Corollary}

\theoremstyle{definition}
\newtheorem{definition}[theorem]{Definition}

\newtheorem{example}[theorem]{Example}

\theoremstyle{remark}

\newtheorem{case[theorem]}{Case}

\def \R{{\mathbb R}}

\def \Z{{\mathbb Z}}
\def \C{{\mathbb C}}

\def\H{{\mathbb H}}

\def\norm#1.#2.{\lVert#1\rVert_{#2}}

\def\R{\mathbb R}

\def \U{{\mathcal U}}

\def \H{{\mathcal H}}
\def \K{{\mathcal K}}
\def \HD{{\mathbb H^d}}
\def \HS{{\mathcal{HS}}}

\title{Positive weight function and classification of g-frames}
 
\author{Anirudha Poria}

\address{Department of Mathematics,
Indian Institute of Science,
Bengaluru 560012, Karnataka, India.}

\address{Department of Mathematics,
School of Engineering and Applied Sciences,
Bennett University, Greater Noida 201310, India.}

\email{anirudhap@iisc.ac.in}



\date{\today}

\begin{document}
\maketitle
\begin{abstract} 
Given a positive weight function and an isometry map on a Hilbert spaces $\H$, we study a class of linear maps which  is a $g$-frame, $g$-Riesz basis and a $g$-orthonormal basis for  $\H$ with respect to $\C$ in terms of the weight function. We apply our results to study the frame for shift-invariant subspaces on the Heisenberg group. 
\end{abstract}    

\section{Introduction and preliminaries} 

A frame for a Hilbert space is a countable set of overcompleted  vectors such that each vector in the Hilbert space  can be represented in a non-unique way in terms of the frame elements. The redundancy and the flexibility in the representation of a Hilbert space vector by the frame elements make the frames a useful tool in mathematics as well as in  interdisciplinary fields such as sigma-delta quantization \cite{ben06}, neural networks \cite{can99}, image processing  \cite{can05}, system modelling \cite{dud98}, quantum measurements \cite{eld02}, sampling theory \cite{fei94}, wireless communications \cite{str03} and many other well known fields.

Given a Hilbert space $\mathcal H$, a countable family of vectors $\{x_j\}_{j\in J}\subset \mathcal H$  is called a {\it frame} for $\mathcal H$ if there are positive constants $A$ and $B$ such that for any $x\in \mathcal H$, 

$$A\|x\|^2 \leq \sum_{j\in J} |\langle x, x_j\rangle|^2 \leq B \|x\|^2 .$$

The frames were introduced for the first time by Duffin and Schaeffer \cite{duf52}, in the context of nonharmonic Fourier series \cite{you01}. The notion of a frame extended to g-frame by Sun \cite{sun06} in 2006 to generalize all the existing frames such as bounded quasi-projectors \cite{for04}, frames of subspaces \cite{cas04}, pseudo-frames \cite{li04}, oblique frames \cite{chr04}, outer frames \cite{ald04}, and time-frequency localization operators \cite{dor06}. Here, we recall the definition of a g-frame. 

\begin{definition}
Let $\mathcal H$ be a Hilbert space and $\{\mathcal K_j\}_{j\in J}$ be a countable family of Hilbert spaces with associated norm $\| \cdot\|_{\mathcal K_j}$.
A countable family  of linear and bounded operators $\{\Lambda_j: \H \to \K_j\}_{j \in J}$ is called a {\it $g$-frame} for $\mathcal H$ with respect to $\{\mathcal K_j\}_{j\in J}$ if there are two positive constants $A$ and $B$ such that for any $f\in \mathcal H$ we have 
\begin{equation}\label{eq01}
 A\|f\|_{\mathcal H}^2 \leq \sum_{j\in J} \|\Lambda_j(f)\|_{\mathcal K_j}^2 \leq B \|f\|_{\mathcal H}^2.
\end{equation}
\end{definition} 
 
The constants $A$ and $B$ are called {\it $g$-frame lower and upper bounds}, respectively. If $A=B=1$, then it is called a {\it Parseval $g$-frame}. For example, by the Riesz representation theorem, every $g$-frame is a frame if $\mathcal K_j= \Bbb C$ for all $j\in J$. And, every frame is a $g$-frame with respect to $\Bbb C$. If the right-hand side of (\ref{eq01}) holds, it is said to be a {\it  $g$-Bessel  sequence} with bound $B$. The family $\{\Lambda_j \}_{j \in J}$ is called {\it  $g$-complete}, if for any vector $f\in \mathcal H$ with $\Lambda_j(f)=0$ for $j \in J$, we have $f=0$. If $\{\Lambda_j \}_{j \in J}$ is $g$-complete and there are positive constants $A$ and $B$ such that for any finite subset $J_1 \subset J$ and $g_j \in \K_j, \; j \in J_1$,
\[ A \sum_{j \in J_1} \|g_j \|^2 \leq \bigg\|\sum_{j \in J_1} \Lambda_j^* (g_j) \bigg\|^2 \leq B \sum_{j \in J_1}\|g_j \|^2, \]
then $\{\Lambda_j \}_{j \in J}$ is called a {\it $g$-Riesz basis} for $\H$ with respect to $\{\K_j\}_{j \in J}$. Here, $\Lambda_j^*$ denotes the adjoint operator. We say $\{\Lambda_j \}_{j \in J}$ is a $g$-orthonormal basis for $\H$ with respect to $\{\K_j\}_{j \in J}$ if it satisfies the following:
\begin{align}\label{onb1}
\langle \Lambda^*_{i} g, \Lambda^*_{j}  h \rangle &= 0 \quad \forall g \in \K_i, \; h \in \K_j, i\neq j \\\label{onb2}
  \| \Lambda^*_{i} g\|^2 &=  \|g \|^2, \quad \forall i, \; g \in \K_i \\\label{onb3}
 \sum_{j \in J}\|\Lambda_jf \|^2 &=\|f \|^2, \quad \forall f \in \H. 
\end{align}

Before we  state the main results of this paper, let us consider the following example. For a function $\phi\in L^2(\Bbb R^d)$, $d\geq 1$ and $m, n\in \Bbb Z^d$, the modulation and translation of $\phi$ by multi-integers $m$ and $n$ are defined by 
 $$M_m\phi(x) = e^{2\pi i \langle m, x\rangle} \phi(x), \quad T_n\phi(x) =\phi(x-n) .$$
The Gabor (Weil-Heisenberg) system generated by $\phi$  is  
 $$\mathcal G(\phi):=\{M_m T_n\phi: \ m, n\in \Bbb Z^d\} .$$
 
It is well-known that the ``basis\rq\rq{} property of  the Gabor system $\mathcal G(\phi)$ for its spanned vector space can be studied by the Zak transform of $\phi$ 

$$Z\phi(x, \xi) = \sum_{k\in \Bbb Z^d} \phi(x+k) e^{2\pi i \langle \xi, k\rangle} .$$

For example, the Gabor system $\mathcal G(\phi)$ is a Riesz basis for the spanned vector space if and only if there are positive constants $A$ and $B$ such that $A\leq |Z\phi(x, \xi)|\leq B$  a.e. $(x,\xi)\in [0,1]^d\times [0,1]^d$ (\cite{HSWW}). The purpose of this paper is to show that the above result is a particular case  of a more general theory involving $g$-frames.
 
For the rest of the paper  we shall assume the following. $\Omega$ is a measurable set with measure $dx$. We assume that $\Omega$ has finite measure $|\Omega|$ and $|\Omega|=1$.   We let $w:\Omega \to (0,\infty)$be a  measurable map  with $\int_\Omega w(x)dx<\infty$.  Let  $\U$ be a Hilbert space over the field $\Bbb F$ ($\Bbb R$ or $\Bbb C$) with associated inner product $\langle \cdot , \cdot \rangle_{\U} $.   We denote by  $L^2_w(\Omega, \U)$  the weighted Hilbert space of all measurable functions $f:\Omega\to \U$ such that 
$$\|f\|_{L^2_w(\Omega, \U)}^2: = \int_\Omega \|f(x)\|_\U^2 w(x) dx<\infty .$$ 

The associated  inner product of any two functions $f, g$ in $L_w^2(\Omega, \U)$ is  then given by 
$$\langle f, g\rangle_{L_w^2(\Omega, \U)} = \int_\Omega \langle f(x), g(x)\rangle_\U w(x) dx. $$  

A countable family of unit vectors  $\{f_k\}_{k\in K}$ in ${L_w^2(\Omega, \U)}$ constitute  an ONB (orthonormal basis) for ${L_w^2(\Omega, \U)}$ with respect to the weight function $w$ if the family  is  orthogonal  and for any $g\in {L_w^2(\Omega, \U)}$ the Parseval identity holds: 
$$\| g \|^2 = \sum_{k\in K} \left|\langle g, f_k\rangle_{L_w^2(\Omega, \U)}\right|^2 =\sum_{k\in K}\left|
\int_\Omega \langle g(x), f_k(x)\rangle_{\U} w(x) dx\right|^2 .$$ 

To avoid any confusion, in the sequel, we shall use subscripts for all inner products and associated norms for Hilbert spaces, when necessary.
For the rest, we assume that $S: \mathcal H \to L^2_w(\Omega, \U)$ is a linear and unitary map. Thus for any $f\in \mathcal H$
$$ \|f\|^2 = \int_\Omega \|Sf(x)\|_{\U}^2 w(x) dx. $$
  
We fix an ONB $\{f_n\}_{n\in K}$ for $L^2(\Omega)$ and ONB $\{g_m\}_{m\in J}$ for the Hilbert space $\U$, and define 
$$G_{m,n}(x):  = f_n(x) g_m, \quad  \forall x\in \Omega, \  (m,n)\in K\times J.$$  
And, 
$$\tilde \Lambda_{(m,n)}(f)(x) = \langle  S(f)(x), G_{(m,n)}(x) \rangle_{\U}  \quad \forall \  f\in \mathcal H,   x\in \Omega .$$

Our main results are the following. 
\begin{theorem}\label{th1} 
Let $\{f_n\}_{n\in K}\subset  L^2(\Omega)$, $\{g_m\}_{m\in J} \subset \U$ and $\{\tilde \Lambda_{m,n}\}$ be as in above.  Assume that $|f_n(x)|=1$ for a.e. $x\in \Omega$.  
Then the following hold: 
\begin{itemize}
\item [(a)]  $\{\tilde\Lambda_{(m,n)}\}_m $  is a Parseval $g$-frame for $\mathcal H$.   Thus $\{\Lambda_{(m,n)}\}_m$ is a Bessel sequence. 
\item [(b)] For any ${(m,n)}$,  the linear map $\Lambda_{m,n}: \mathcal H \to \C$ defined  by
$$\Lambda_{m,n}(f) =\int_\Omega \tilde\Lambda_{m,n}(f)(x) \  w(x) dx $$
is well-defined. And,  $\{\Lambda_{m,n}\}$   is a frame for $\H$   if and only if there are positive finite constants $A$ and $B$ such that  $A\leq w(x)\leq B$ for a.e. $x\in \Omega$.
\item [(c)] The family $\{\Lambda_{m,n}\}$ is a Riesz basis for $\H$ if and only if there are positive finite constants $A$ and $B$ such that $A\leq w(x)\leq B$ for a.e. $x\in \Omega$. 
\end{itemize} 
\end{theorem}  

\begin{corollary}\label{main result}  
Let $\{\lambda_k\}_{k\in J}$ be an orthonormal basis (or a Parseval frame) for $L^2(\Omega)$ such that $|\lambda_k(x)|=1$ for all $x\in \Omega$. Assume that $S: \mathcal H \to L^2_w(\Omega)$ is a  unitary map. Then the sequence of operators $\{\Lambda_k\}_{k\in J}$ defined by 
$$\Lambda_k(f) =\int_\Omega  S(f)(x) \overline{ \lambda_k(x) } w(x) dx $$
is a frame for $\H$   if and only if there are positive finite constants $A$ and $B$ such that  $A\leq w(x)\leq B$ for a.e. $x\in \Omega$. 
\end{corollary} 

\begin{theorem}\label{th3} 
Let $\{f_n\}_{n\in K}\subset  L^2(\Omega)$, $\{g_m\}_{m\in J} \subset \U$ and $\{\tilde \Lambda_{m,n}\}$ be as in above. Assume that $|f_n(x)|=1$ for a.e. $x\in \Omega$. The family $\{\Lambda_{m,n}\}$ is an ONB for $\H$ if and only if $w(x)=1$ for a.e. $x\in \Omega$. 
\end{theorem}  

\section{Proof of Theorem \ref{th1}} 
First we prove the following lemmas which we  need for the proof of Theorem \ref{th1}. 
 
\begin{lemma}\label{technical lemma}
Let $\{f_n\}$ be an ONB for the weighted Hilbert space $L_w^2(\Omega)$. Let $\{g_m\}$ be an ONB for a Hilbert space $\U$. Define $G_{m,n}(x) = f_n(x) g_m$. Then the family $\{G_{m,n}\}_{J\times K}$ is an ONB for $L_w^2(\Omega, \U)$. 
\end{lemma} 

In order to prove the lemma, we shall recall the following result from \cite{iosevich-mayeli-14}  and prove it here for the sake of completeness.  

\begin{lemma}\label{mixed orthonormal bases} 
Let $(X,\mu)$ be a measurable space,  and $\{f_n\}_n$ be an orthonormal basis for $L^2(X):=L^2(X, d\mu)$. Let $Y$ be a Hilbert space and $\{g_m\}_m$ be a  family in  $Y$. For any $m, n$ and $x\in X$ define $G_{m,n}(x) := f_n(x) g_m$. Then $\{G_{m,n}\}_{m,n}$ is an orthonormal basis for the Hilbert space $L^2(X,Y,d\mu)$ if and if $\{g_m\}_m$ is an orthonormal basis for $Y$. 
\end{lemma}

\begin{proof}  
For any $m, n$ and $m\rq{}, n\rq{}$ we have  
\begin{align}\label{orthogonality-relation}
\langle G_{m,n}, G_{m\rq{},n\rq{}} \rangle &= \int_X \langle f_m(x) g_n, f_{m\rq{}}(x) g_{n\rq{}} \rangle_Y \  d\mu(x)\\\notag
& = \langle f_m,f_{m\rq{}}\rangle_{L^2(X)} \langle g_n, g_{n\rq{}}\rangle_Y \\\notag
&= \delta_{m,m\rq{}}  \langle g_n, g_{n\rq{}}\rangle_Y.
\end{align}
 
This shows that the orthogonality of  $\{G_{m,n}\}_{m,n}$  is equivalent to the orthogonality of $\{g_m\}_m$. And,  $\|G_{m,n}\| =1$ if and only if $\|g_n\|=1$.  
  
Let $\{g_m\}_m$ be an orthonormal basis for $Y$. To prove the completeness of  $\{G_{m,n}\}$ in $L^2(X,Y,d\mu)$, let $F\in  L^2(X,Y,d\mu)$ such that $\langle F, G_{m,n}\rangle =0$, $\forall \ m, n$. We claim $F=0$. By the definition of the inner product we have 
\begin{align}\label{inner-product}
0=\langle F, G_{m,n}\rangle  &=\int_X \langle F(x), G_{m,n}(x)\rangle_Y d\mu(x)\\\notag 
&= \int_X \langle F(x), f_m(x) g_n\rangle_Y d\mu(x) \\\notag
&= \int_X \langle F(x), g_n\rangle_Y \overline{f_m(x)} d\mu(x) \\\notag
&= \langle A_n, f_m\rangle
\end{align}
where  
$$A_n: X\to \C; \ \  x \mapsto \langle F(x), g_n\rangle_Y. $$

$A_n$ is a measurable function and lies in $L^2(X)$ with $\|A_n\|\leq  \|F\|$. Since $\langle A_n, f_m\rangle_{L^2(X)}=0$ for all $m$, then  $A_n=0$ by the completeness of $\{f_m\}$.  On the other hand, by  the definition of $A_n$  we have $\langle F(x), g_n\rangle_{Y}=0$ for a.e. $x\in X$. Since $\{g_n\}$ is complete in $Y$, then $F(x)=0$ for a.e. $x\in X$. This proves the claim. 
   
Conversely, assume that $\{G_{m,n}\}_{m,n}$ is an orthonormal basis for the Hilbert space $ L^2(X,Y,d\mu)$. Therefore by (\ref{orthogonality-relation}), $\{g_m\}$ is an  orthonormal set. We prove that if for $g\in Y$ and $\langle g, g_m\rangle=0$ for all $m$, then $g$ must be identical to zero. For this, for any $n$ define the map 
$$B_n: X\to Y; \ \ x\mapsto f_n(x)g.$$
Then $B_n$ is measurable and it belongs to $ L^2(X,Y,d\mu)$ and  $\|B_n\|  = \|g\|_Y$. Thus  
\begin{align}
B_n&=\sum_{n\rq{},m} \langle B_n, G_{m,n\rq{}}\rangle_{L^2(X,Y,d\mu)} G_{m,n\rq{}}\\\notag
&= \sum_{n\rq{},m} \langle f_n, f_{n\rq{}}\rangle_{L^2(X)} \langle g, g_m\rangle_Y G_{m,n\rq{}}\\\notag
&= \sum_{m}   \langle g, g_m\rangle_Y G_{n,m}.
\end{align}
By the assumption that $ \langle g, g_m\rangle_Y=0$ for all $m$,  we get $B_n=0$. This implies that $B_n(x)= f_n(x) g=0$ for a.e. $x$. Since, $f_n\neq 0$, then $g$ must be a zero vector, and hence we are done. 
\end{proof}

\begin{lemma}\label{boundedness} 
$\tilde \Lambda_{(m,n)}: \mathcal H\to L^2_w(\Omega)$ is a bounded operator and $\|\tilde \Lambda_{(m,n)}(f)\|_{L^2_w(\Omega)}\leq \|f\|$. 
\end{lemma}

\begin{proof}
Let $f\in \H$. Then  for any $m\in J$ and $n\in K$,
\begin{align*}
\| \tilde\Lambda_{m,n}(f)\|_{L^2_w(\Omega)}^2 &= \int_\Omega |\tilde\Lambda_{m,n}(f)(x)|^2 w(x) dx \\
 &= \int_\Omega |\langle Sf(x), f_n(x)g_m\rangle_{\U}|^2 w(x) dx\\
 &= \int_\Omega |\langle Sf(x), g_m\rangle_{\U}|^2 w(x) dx .
\end{align*}
Using the Cauchy--Schwartz inequality in the preceding line, we get 
\begin{align*}
\| \tilde\Lambda_{m,n}(f)\|_{L^2_w(\Omega)}^2 &\leq  \int_\Omega \|Sf(x)\|^2 w(x) dx =  \|f\|^2 . 
\end{align*} 
\end{proof} 

Here, we first calculate the adjoint of $\tilde\Lambda_{m,n}$: $S$ is a unitary map. Then for any $f\in \H$ and $h\in L_w^2(\Omega,\U)$ we have 
\begin{align} 
\int_\Omega \langle Sf(x), h(x)\rangle_{\U} w(x)dx= \langle Sf, h\rangle_{L_w^2(\Omega,\U)}  =\langle f, S^{-1}h\rangle .
\end{align} 
Therefore for any $\phi\in L_w^2(\Omega)$ we get 
\begin{align}\label{adjoint}
 \langle \tilde \Lambda_{m,n} f, \phi\rangle = \langle f, S^{-1}((f_n \phi)g_m)\rangle ,
\end{align} 
where $(f_n \phi)g_m \in L_w^2(\Omega, \U)$ and $(f_n \phi)g_m(x) = f_n(x)\phi(x) g_m$. The relation (\ref{adjoint}) indicates that 
$$\tilde \Lambda_{m,n}^*(\phi) = S^{-1}((f_n \phi)g_m).$$ 
Notice, for any $f\in \H$, 
$$\Lambda_{m,n} f = \langle f, S^{-1}(f_ng_m)\rangle_{\H} .$$
Thus $\Lambda_{m,n}^*: \C \to \H$ is given by $c\to cS^{-1}(f_ng_m)$.

\begin{proof}[Proof of Theorem \ref{th1}]
$(a)$: Observe that $|\tilde \Lambda_{m,n}(f)(x)|\leq \| Sf(x)\|$ and $S$ is an isometry map. For any $f \in \H$ and $n\in K$  we have
\begin{align*}
\sum_m \| \tilde \Lambda_{m,n} (f)\|^2_{L^2_w(\Omega)} & =\sum_m \int_{\Omega} | \tilde \Lambda_{m,n} (f)(x)|^2 w(x) dx \\
& = \sum_m \int_{\Omega} | \langle G_{m,n}(x) , S(f)(x)\rangle_{\U} |^2 w(x) dx \\ 
& = \int_{\Omega} \sum_m| \langle G_{m,n}(x) , S(f)(x)\rangle_{\U} |^2 w(x) dx \\
& = \int_{\Omega} \sum_m| \langle g_m , w^{1/2}S(f)(x)\rangle_{\U} |^2   dx. 
\end{align*}
By the assumptions of the theorem, for   a.e. $x\in \Omega$, the sequence    $\{g_m\}_m$ is an ONB for $\U$. Invoking this along the isometry property of $S$ in  the last equation above, we get 
\begin{align}\label{parseval-proprty}
\sum_m \| \tilde \Lambda_{m,n} (f)\|^2_{L^2_w(\Omega)}   
 = \int_{\Omega} \| S(f)(x) \|^2_{\U} \; w(x) dx  
 = \| f \|^2.
\end{align}
Therefore, $\{\tilde \Lambda_{m,n}\}_m$ is a Parseval $g$-frame for  $\H$ with respect to $L^2_w(\Omega)$. To prove that $\{\Lambda_{m,n}\}_m$ is a Bessel sequence, note that by the H{\"o}lder's inequality for  in the weighted Hilbert space $L^2_w(\Omega)$  we can write 
\begin{align*}
|\Lambda_{m,n}(f)| \leq \int_\Omega |\tilde \Lambda_{m,n} (f)(x)| w(x) dx & \leq \bigg( \int_\Omega |\tilde \Lambda_{m,n} (f)(x)|^2 w(x) dx \bigg)^{\frac{1}{2}}   \bigg( \int_\Omega w(x) dx \bigg)^{\frac{1}{2}}.  
\end{align*}
By Lemma \ref{boundedness}, the first integral on the right is finite. Therefore by summing the square of the terms over $m$ we get 
\[ \sum_{m \in J} |\Lambda_{m,n}(f)|^2 \leq  C \sum_{m \in J} \int_\Omega |\tilde \Lambda_{m,n} (f)(x)  |^2 w(x) dx =C \sum_{m \in J} \| \tilde \Lambda_{m,n} (f) \|_{L^2_w(\Omega)}^2 =C\| f \|^2, \]
where $C:= \int_\Omega w(x) dx$ is a non-zero constant. Thus $\{\Lambda_{m,n}\}_{m \in J}$ is a  Bessel sequence for $\H$ with bound $C$. Notice, in the last equality we used (\ref{parseval-proprty}).

$(b)$ The map $\Lambda_{m,n}:\H \to \C$ is linear, well-defined and bounded. Indeed, for any $f\in \H$,
\begin{align*}
\int_\Omega |\tilde \Lambda_{m,n} (f)(x)| w(x) dx &\leq \left(\int_\Omega w(x) dx\right)^{1/2}  \left(\int_\Omega\|Sf(x)\|^2 w(x) dx\right)^{1/2} \\
&= \|f\| \left(\int_\Omega w(x) dx\right)^{1/2}.   
\end{align*}
Assume that $A\leq w(x)\leq B$ for almost every $x\in \Omega$. Let  $f\in \mathcal H$. Then 
\begin{align}\notag
\sum_{m,n} |\Lambda_{m,n}(f)|^2 &= \sum_{m,n} \left|  \int_\Omega \langle G_{m,n}(x) , S(f)(x)\rangle_{\U} w(x) dx\right|^2 \\\notag 
&=  \sum_{m,n}  \left|  \int_\Omega \langle G_{m,n}(x) , S(f)(x)w(x)\rangle_{\U} dx\right|^2  \\\label{eq1}
&=\sum_{m,n} \left| \langle  G_{m,n}  , S(f)w\rangle_{L^2(\Omega, \U)}\right|^2 .
\end{align} 
Since the countable family  $\{ G_{m,n}\}_{m,n}$ is an ONB for $L^2(\Omega, \U)$. Thus   
\begin{align}\notag 
(\ref{eq1})&=  \Vert S(f) w \Vert^2_{L^2(\Omega, \U)} \\\label{equ2}
& =  \int_\Omega \Vert S(f)(x)\Vert^2_{\U} \ w(x)^2  dx. 
\end{align}
By invoking the assumption that $w(x)\leq B$ for a.e. $x\in \Omega$ in (\ref{equ2}) we obtain 
\begin{align}\notag
(\ref{equ2})
 \leq B \int_\Omega \Vert S(f)(x)\Vert^2_{\U} \ w(x)  dx  
 =  B \|S(f)\|^2_{L^2_w(\Omega, \U)}= B \| f\|^2. 
\end{align} 
This proves that the sequence $\{\Lambda_{m,n}\}_{m,n}$ is a Bessel sequence for $\H$. An analogues argument also proves  the  frame  lower bound condition for $\{\Lambda_{m,n}\}_{m,n}$.

For the converse, assume that $\{\Lambda_{m,n}\}_{m,n}$ is a frame for $\mathcal H$ with the frame bounds $0<A\leq B<\infty$. Therefore for any $f\in \mathcal H$  
$$A\|f\|^2 \leq \sum_{m,n} |\Lambda_{m,n}(f)|^2 \leq B \|f\|^2.$$
Assume that there is a set $E\subset \Omega$ with positive measure such that $w(x)<A$ for all $x\in E$. We will prove that  there exits a function in $\mathcal H$ for which  the lower frame condition dose not hold. To this end, let $e_0 \in \U$ be a unit vector and let   $\vec{0}$ denote the zero vector in $\U$. Define $\chi_E(x):= 1_E(x) e_0$. By the assumption, $w \in L^1(E)$. Thus $\chi_E \in L^2_w(\Omega, \U)$. $S$ is  unitary, therefore  there is a function $\phi_E\in \H$ such that $S(\phi_E)=\chi_E$, and we have 
\begin{align}\label{iso}
\| S(\phi_E) \|_{L^2_w(\Omega, \U)}= \|\phi_E \|_{\H}=\| \chi_E \|_{L^2_w(\Omega, \U)}.  
\end{align}
On the other hand, the sequence $\{G_{m,n}\}_{m,n}$ is an ONB for $L^2(\Omega, \U)$. Thus
\begin{align*}
\sum_{m,n} |\Lambda_{m,n}(\phi_E)|^2 &= \sum_{m,n} \left| \langle G_{m,n}, S(\phi_E)w \rangle_{L^2(\Omega, \U)} \right|^2 \\
&=  \Vert \chi_E w  \Vert^2_{L^2(\Omega, \U)} \\
& =  \int_\Omega \Vert \chi_E(x)  \Vert^2_{\U} \ w(x)^2  dx \\
&< A \int_\Omega \Vert \chi_E(x) \Vert^2_{\U} \ w(x)  dx \\
& = A  \|\chi_E\|^2_{L_w^2(\Omega,\U)}  \\
&= A \| \phi_E \|^2_{\H} \hspace{1in}  \text{by \  (\ref{iso})}. 
\end{align*}
The preceding calculation shows that the lower frame bound condition fails for $\phi_E$. This contradicts our assumption that $\{\Lambda_{m,n}\}_{m,n}$ is a frame for $\H$, therefore $w(x)\geq A$ a.e. $x\in \Omega$. The argument for the upper bound for $w$ follows similarly. 

$(c)$ Assume that $A\leq w(x)\leq B$ for almost every $x\in \Omega$. Let $\{c_{m,n}\}_{m,n}$ be any finite sequence in $\C$. Then 
\begin{align*} 
 \left\| \sum_{m,n} \Lambda_{m,n}^*(c_{m,n}) \right\|_\H^2 &= \left\| \sum_{m,n}  c_{m,n}  S^{-1}(f_ng_m) \right\|_\H^2 \\
 & = \left\| S^{-1}\left(\sum_{m,n}  c_{m,n}   f_ng_m\right) \right\|_\H^2 \\
 & = \left\|   \sum_{m,n}  c_{m,n}   f_ng_m\right\|_{L^2_w(\Omega,\U)}^2  & \text{($S$ is unitary)}\\
 &= \int_\Omega  \left\|   \sum_{m,n}  c_{m,n}   f_n(x)g_m\right\|_{\U}^2 w(x) dx\\
 & = \int_\Omega  \sum_{m,n}  |c_{m,n}|^2     w(x) dx    & \text{(by orthogonality of $g_m$)}\\
 &\leq B   \sum_{m,n}  |c_{m,n}|^2  &\text{(since $w(x)\leq B$ a.e. $x\in \Omega$).}
\end{align*}
We also have 
\begin{align*} 
\left\| \sum_{m,n} \Lambda_{m,n}^*(c_{m,n}) \right\|_\H^2 & = 
\int_\Omega  \sum_{m,n}  |c_{m,n}|^2 w(x) dx  \geq   A  \sum_{m,n}  |c_{m,n}|^2  &\text{(since $w(x) \geq A$ a.e. $x\in \Omega$).}
\end{align*}
These show that $\{\Lambda_{m,n}\}_{K\times J}$ is a Riesz basis for $\H$ with lower and upper Riesz bounds $A$ and $B$, respectively.
 
Now assume that $\{\Lambda_{m,n}\}_{K\times J}$ is a Riesz basis for $\H$ with Riesz bounds $A$ and $B$. Therefore, for any sequence $\{c_{m,n}\}_{m,n}$ the inequalities hold: 
\begin{align}\label{Riesz inequality}
 A\sum_{m,n}  |c_{m,n}|^2  \leq  \left\| \sum_{m,n}\Lambda_{m,n}^*(c_{m,n})\right\|^2 \leq B \sum_{m,n}  |c_{m,n}|^2. 
\end{align}
 
We show that there are positive constants $A$ and $B$ such that $A\leq w(x)\leq B$ for a.e. $x\in \Omega$. In contrary, without loss of generality, we assume then there is a measurable subset $E\subset \Omega$ with positive measure such that $w(x)<A$ for all $x\in E$. Let $e$ be any unitary vector in the Hilbert space $\U$ and define the function ${\bf 1}_E(x) = e$ if $x\in E$ and otherwise ${\bf 1}_E(x)=0$. It is clear that ${\bf 1}_E\in L^2(\Omega, \U)$. Thus, there are coefficients $\{c_{m,n}\}_{K\times J}$ such that ${\bf 1}_E= \sum_{m,n} c_{m,n} f_n g_m$ and $\|{\bf 1}_E\|^2= \sum_{m,n}|c_{m,n}|^2$.  Then $\|{\bf 1}_E(x)\|=1$ for all $x\in E$ and we get
\begin{align*}
\left\| \sum_{m,n}\Lambda_{m,n}^*(c_{m,n})\right\|^2 &=   \left\| \sum_{m,n}  c_{m,n} f_ng_m\right\|_{L^2_w(\Omega,\U)}^2 \\
&= \huge\int_\Omega  \left\| \sum_{m,n} c_{m,n} f_n(x)g_m\right\|_{\U}^2 w(x) dx\\
&= \huge\int_\Omega \left\| {\bf 1}_E(x)\right\|_{\U}^2 w(x) dx\\
&= \huge\int_E  w(x) dx\\
&\leq A \int_\Omega \left\| {\bf 1}_E(x)\right\|_{\U}^2 dx \quad \text{(by the assumption $w(x)<A$ for all  $x\in E$)}\\
&=  A \int_\Omega \left\| \sum_{m,n}  c_{m,n} f_n(x)g_m\right\|_{\U}^2 dx \\
&=  A \int_\Omega \sum_{m,n} |c_{m,n}|^2 dx \\
&=  A  \sum_{m,n} |c_{m,n}|^2 . 
\end{align*}
This is contrary to the lower bound condition in (\ref{Riesz inequality}).
\end{proof} 

\section{Proof of Corollary \ref{main result}} 

\begin{proof} 
Assume that $A\leq w(x)\leq B$ for almost every $x\in \Omega$. Let $f\in \mathcal H$. Then 
\begin{align}\label{first line} 
\sum_{k \in J} |\Lambda_k(f)|^2 &= \sum_{k \in J} \left| \int_\Omega  S(f)(x) \overline{\lambda_k(x)} w(x)  dx\right|^2. 
\end{align} 
By the fact that $\{\lambda_k\}_{k \in J}$ is an orthonormal basis for $L^2(\Omega)$, using the Plancherel\rq{}s theorem we continue as follows: 
\begin{align*} 
(\ref{first line})&=  \int_\Omega |S(f)(x) w(x)|^2 dx 
\leq B \int_\Omega |S(f)(x)|^2 w(x) dx 
= B \|f\|^2. 
\end{align*} 
The frame boundedness from below by $A$ also follows with a similar calculation. 

For the converse, we shall use a contradiction argument. Assume that $\{\Lambda_k\}_{k \in J}$ is a frame for $\mathcal H$ with the frame bounds $0<A\leq B<\infty$. Therefore for any $f\in \mathcal H$ we have 
$$A\|f\|^2 \leq \sum_{k \in J} |\Lambda_k(f)|^2 \leq B \|f\|^2.$$

Assume that $E\subset \Omega$ be a measurable set with positive measure such that $w(x)< A$ for all $x\in E$. We shall show that the lower bound condition for the frame $\{\Lambda_k\}_{k \in J}$ must then fail for the lower bound $A$.

By the assumptions, $w \in L^1(E)$. Thus $1_E\in L^2_w(\Omega)$. Since $S$ is an onto map, assume that $\phi_E$ is the pre-image of $1_E$ in $\mathcal H$. Therefore  $S(\phi_E) = 1_E$ and $\|\phi_E\|_\mathcal H = \|1_E\|_{L^2_w(\Omega)}$ and we have 
\begin{align}\label{E1}
\sum_{k \in J} |\Lambda_k(\phi_E)|^2 = \sum_{k \in J} \left|\int_\Omega  S(\phi_E)(x) \overline{\lambda_k(x)} w(x)dx \right|^2
= \int_\Omega |1_E(x)|^2 |w(x)|^2 dx
= \int_E |w(x)|^2 dx.   
\end{align}
Since $w(x)<A$ for all $x\in E$, then from the last integral we obtain the following: 
\begin{align*}
(\ref{E1}) \leq A \int_E w(x) dx  
= A  \int_\Omega  |1_E(x)|^2 w(x) dx  
= A\int_\Omega |S(\phi_E)(x)|^2 w(x) dx  
= A\|\phi_E\|^2.
\end{align*} 
The preceding calculation shows that the frame lower bound condition fails for $\phi_E$. This contradicts our assumption that $\{\Lambda_k\}_{k \in J}$ is a frame, therefore $w(x)\geq A$ a.e. $x\in \Omega$. The argument for the upper bound follows similarly.
\end{proof} 

\section{Proof of Theorem \ref{th3}} 
\begin{proof} 
Assume that $w(x)=1$ a.e. $x\in \Omega$. By the equations (\ref{eq1}) and (\ref{equ2}), for any $f\in \H$ we have 
\begin{align}\notag
\sum_{m,n} |\Lambda_{m,n}(f)|^2 
&=\sum_{m,n} \left| \langle  G_{m,n} , S(f)\rangle_{L^2(\Omega, \U)}\right|^2 = \| Sf\|^2 = \|f\|^2 .
\end{align} 
This proves (\ref{onb3}). Let $c_1, c_2\in \C$. For any $m,m\rq{}\in J$ and $n, n\rq{}\in K$ we have 
 \begin{align*} 
 \langle \Lambda_{m,n}^*(c_1),\Lambda_{m\rq{},n\rq{}}^*(c_2)\rangle &= c_1\overline{c_2} \langle S^{-1}(f_ng_m),  S^{-1}(f_{n\rq{}}g_{m\rq{}})\rangle_{\H} \\
 & = c_1\overline{c_2}  \langle f_ng_m,  f_{n\rq{}}g_{m\rq{}}\rangle_{L^2(\Omega, \U)} \\
 &=  c_1\overline{c_2}  \delta_{m,m\rq{}}\delta_{n,n\rq{}}.
 \end{align*}
This proves the relations (\ref{onb1}) and (\ref{onb2}).  
 
To prove the converse, we  assume contrary.  We assume that there is a subset $E\subset \Omega$ of positive measure for which $w(x)<1$. As in the proof of Theorem \ref{th1}, one can show the existence of a function $\phi_E$ for which  with an analogous calculation following the relation (\ref{iso}) the following holds: 
$$\sum_{m,n} |\Lambda_{m,n}(\phi_E)|^2  \leq \|\phi_E\|^2 .$$
This indicates that the relation (\ref{onb3}) does not hold for $\phi_E$, which contradicts the assumption. 
\end{proof} 

\section{Examples} 
\begin{example}\label{example 1}  
Let  $\Omega=D$ be a  fundamental domain  in $\Bbb R^d$ with Lebesgue measure one. Assume that $\Gamma=M\Bbb Z^d$  where $M$ is an $d\times d$ invertible matrix and the exponentials $\{e_n(x):= e^{2\pi i \langle n,x\rangle}: \ n\in \Gamma\}$ form an orthonormal basis for $L^2(D)$. 

For $0\neq\phi\in L^2(\Bbb R^d)$, define  $\H:= \overline{\text{span}\{\phi(\cdot-n): n\in  \Bbb Z^d\}}$ and  the weight function $w$ by $w(x) := \sum_{n\in \Gamma^\perp} |\hat\phi(x+n)|^2$, a.e. $x\in D$. We claim that $\int_D w(x) dx= \|\phi\|^2$, thus is finite. To this end, notice   $D$ is a fundamental domain for the lattice $\Gamma$. By a result by Fuglede \cite{Fug74}, $D$ tiles $\Bbb R^d$ by  the dual  lattice $\Gamma^\perp=M^{-t}\Bbb Z^d$, $M^{-t}$ the inverse transpose of $M$. Therefore, we have 
\begin{align*}
\int_D w(x) dx &= \int_D \sum_{n\in \Gamma^\perp} |\hat\phi(x+n)|^2 dx  
= \int_{\R^d} |\hat\phi(x)|^2 dx 
= \|\hat \phi\|^2  
= \|\phi\|^2 .
\end{align*}
Let $E_w:= \{x\in D: w(x)>0\}$ and for any $f\in \H$ define 
$$S(f)(x):= 1_{E_w}(x)  {w(x)}^{-1} \sum_{n\in \Gamma^\perp} \hat f(x+n) \overline{\hat \phi(x+n)} \quad   \text{a.e.}\   x\in E_w.$$
Then $S$ is an unitary map from  $\H$ onto the weighted Hilbert space $L_w^2(0,1)$ with $\U=\C$. Note that    $Sf(x)= 0$ a.e. $x\in  D\setminus E_w$ (Theorem 3.1 (i) \cite{HSWW}).

For $k\in \Gamma$ define 
$$\Lambda_k(f):= \int_{E_w} \left(\sum_{n\in \Gamma^\perp} \hat f(x+n)\overline{\hat\phi(x+n)}\right) e^{-2\pi i  \langle x, k\rangle } dx.$$ 
By Corollary \ref{main result}, the operators $\{\Lambda_k\}_{k\in \Gamma}$  constitute a frame for $\H$ if there are positive constants $A$ and $B$ such that $A\leq \sum_{n\in \Gamma^\perp} |\hat\phi(x+n)|^2\leq B$ a.e. $x\in E$. By the well-known periodization method, it is obvious that $\Lambda_k(f) = \langle T_k \phi, f\rangle$ for any $f\in \H$ and  $k\in \Gamma$, with $T_k\phi(x)= \phi(x-k)$. Thus, $\{\Lambda_k\}_\Gamma$ is the translation family $\{T_k \phi\}_\Gamma$. 

For example, if  $\phi\in L^2(\R)$  such that  $\hat\phi= 1_{[0,1]}$, the indicator function of the unit interval, then the inequalities for $w$ holds for $A=1$ and $B=2$, hence   $\{\Lambda_k\}_{k\in \Gamma}$ is a frame with lower and upper frame bounds $1$ and $2$, respectively. 
\end{example} 

\section{Application: Frames for shift-invariant subspaces on the Heisenberg group}
 
In this section we shall revisit the example  of a function in $L^2(\Bbb H^d)$ that was introduced in \cite{BHM14} and    exploit our current  results  to study the frame and Riesz   property of the lattice  translations of the function for a shift-invariant subspace of $L^2(\Bbb H^d)$. 
 
\subsection{The Heisenberg group} 
The $d$-dimensional Heisenberg group $\Bbb H^d$ is identified with $\R^d \times \R^d \times \R$ and the noncommutative group law is  given by
\begin{equation}\label{equ:Hlaw}
(p,q,t) (p',q',t') = (p + p', q + q', t + t' + p\cdot q').
\end{equation}
The inverse  of an element is given by $(p,q,t)^{-1}= (-p,-q, -t+p\cdot q)$. Here, $x\cdot y$ is the inner product of two vectors in $\R^d$. The Haar measure of the group is the Lebesgue measure on $\Bbb R^{2d+1}$.

The class of non-zero measure  irreducible representations of $\Bbb H^d$ is identified by non-zero elements $\lambda \in \R^*:=\R \setminus \{0\}$ (see \cite{F1995}). Indeed, for any  $\lambda\neq 0$, the associated irreducible representation $\rho_\lambda$ of the Heisenberg group  is equivalent to Schr\"odinger representation into the class of unitary operators on $L^2(\Bbb R^d)$, such that for any $(p,q,t)\in \Bbb H^d$ and $f\in L^2(\Bbb R^d)$ 
\begin{align}\label{definition-of-schroedinger-representation}
\rho_\lambda(p,q,t)f(x) = e^{2\pi i t \lambda} e^{-2\pi i \lambda \langle q\cdot x\rangle} f(x-p) . 
\end{align}

Notice $\rho_\lambda(p,q,0)f(x) = M_{\lambda q} T_p f(x)$ is the unitary frequency-translation operator, where $M_x$ and $T_y$ are the modulation and translation operators, respectively. For $\varphi\in L^2(\Bbb H^d)$, we denote by $\hat \varphi$ the operator valued Fourier transform  of $\varphi$  which is defined by 
\begin{equation}\label{equ:Hfourier}
\hat \varphi(\lambda) = \int_\HD \varphi(x) \rho_\lambda(x) dx \quad \forall \lambda\in  \R \setminus \{0\} .
\end{equation}

The operator $\hat \varphi(\lambda)$ is a Hilbert-Schmidt operators on $L^2(\Bbb R^d)$ such that for any $f\in L^2(\Bbb R^d)$ 
\begin{equation}\notag
\hat \varphi(\lambda)f(y) = \int_\HD \varphi(x) \rho_\lambda(x)f(y)  \ dx \quad \forall \lambda\in  \R \setminus \{0\} ,
\end{equation}
and the equality is understood in $L^2$-norm sense. 
  
For any $\psi$ and $\varphi$ in $L^2(\Bbb H^d)$ and $\lambda \in \mathbb R \setminus \{0\}$, the Hilbert-Schmidt inner product $\langle  \hat \varphi(\lambda), \hat\psi(\lambda)\rangle_{\HS}$ is the trace of an operator. Indeed,
\begin{equation}\label{eq:HS}
\langle\hat\varphi(\lambda), \hat\psi(\lambda)\rangle_{\HS} = \textnormal{trace}_{L^2(\R^d)} \left(\hat\varphi (\lambda)  \hat\psi(\lambda)^*\right) .
\end{equation}
(Here, $\hat\psi(\lambda)^*$ denotes  the $L^2(\R^d)$ adjoint of the operator $\hat\psi(\lambda)$.)  It is easy to see that $\hat\varphi (\lambda)  \hat\psi(\lambda)^*$ is a kernel operator. Thus $\langle\hat\varphi(\lambda), \hat\psi(\lambda)\rangle_{\HS}$ is trace of a kernel operator (\cite{F1995}). The Plancherel formula for the Heisenberg group is given by 
\begin{equation}\label{eq:Plancherel}
\langle \varphi, \psi\rangle_{L^2(\HD)} = \int_\R   \langle\hat\varphi(\lambda), \hat\psi(\lambda)\rangle_{\HS}  |\lambda|^d d\lambda . 
\end{equation}
The measure $|\lambda|^d d\lambda$ is the Plancherel measure on the non-zero measure class of irreducible representations of the Heisenberg group (\cite{F1995}) and $d\lambda$ is the Lebesgue measure on $\Bbb R^*$. By the periodization method, the integral in (\ref{eq:Plancherel}) is can be equivalently written as 
\begin{equation}\notag
 \int_\R   \langle\hat\varphi(\lambda), \hat\psi(\lambda)\rangle_{\HS}  |\lambda|^d d\lambda  = \int_0^1  
 \sum_{j\in \Bbb Z}  \langle\hat\varphi(\alpha+j), \hat\psi(\alpha+j)\rangle_{\HS}  |\alpha+j|^d d\alpha . 
\end{equation} 
Thus, for any $\varphi\in L^2(\Bbb H^d)$, by the Plancherel formula we deduce the following:
\begin{equation}\label{periodization}
\|\varphi\|^2 = \int_0^1  \sum_{j\in \Bbb Z}  \|\hat \varphi(\alpha+j)\|^2_{\HS}  |\alpha+j|^d d\alpha . 
\end{equation} 

\subsection{Frames for  a shift-invariant subspace} 
Let $u={\bf 1}_{[0,1]^d}\in L^2(\Bbb R^d)$ be the indicator function of the unit cube $[0,1]^d$. For $\alpha\neq 0$, define the $L^2$-unitary dilation of $u$  with respect to $\alpha$  by 
$u_\alpha(x)= |\alpha|^{-d/2}u(\alpha x)$. Let $a$ and $b$ be two real numbers such that $0\neq ab\in \Bbb Z$. Then the family consisted of translations and modulations of $u_\alpha$ by $a\Bbb Z^d$ and $b \Bbb Z^d$, respectively, is then given by 
$$ \left\{|\alpha|^{d/2}e^{-2\pi i\alpha b\langle m, x\rangle} {\bf 1}_{(0,\frac{1}{\alpha})^d}(x-an): \  \ m, n\in \Bbb Z^d\right\} .$$
It is known that  the family is  an orthonormal basis for $L^2(\Bbb R^d)$ and is called  orthonormal  Gabor or Weyl-Heisenberg basis for $L^2(\Bbb R^d)$ with the window function $u_\alpha$. 
   
Fix  $0<\epsilon<1$ (for the rest of the paper)  and  define the  projector map $\Psi_\epsilon$ from  $(0,1)$ into the class  of Hilbert-Schmidt operators of rank one on $L^2(\Bbb R^d)$  by 
$$\Psi_\epsilon(\alpha) := (u_\alpha\otimes u_\alpha) 1_{(\epsilon,1]}(\alpha),$$
where for any $f, g, h\in L^2(\Bbb R^d)$ we have $(f\otimes g)h:= \langle h, g\rangle f$. 
By the definition of the Hilbert-Schmidt norm, we then have 
\begin{align}\label{value of HS-norm}
\|\Psi_\epsilon(\alpha)\|_{\mathcal{HS}} = 1_{(\epsilon,1]}(\alpha) .  
\end{align}  
Thus $\|\Psi_\epsilon\|^2 = (d+1)^{-1}(1-\epsilon^{d+1})$. This implies that $\Psi_\epsilon\in  L^2(\R^*,\HS(L^2(\R^d)),|\lambda|^d d\lambda)$. Therefore, by the inverse Fourier transform for the Heisenberg group, there is a function in $L^2(\Bbb H^d)$ whose Fourier transform is identical to $\Psi_\epsilon$ in $L^2$-norm. We let $\psi_\epsilon$  denote this function $ L^2(\Bbb H^d)$. 
 
Let $\Bbb A$ and $\Bbb B$ be any $d\times d$ matrices in $GL(\Bbb R, d)$ such that $\Bbb A\Bbb B^t\in GL(\Bbb Z,d)$.  Define $\Gamma :=\Bbb A\Bbb Z^d \times \Bbb B \Bbb Z^d\times\Bbb Z$. Then  $\Gamma$ is a lattice subgroup of the Heisenberg group, a discrete and co-compact subgroup.  For any $\gamma=(p,q,t)\in \Gamma$, we denote by $T_\gamma \psi_{\epsilon}$ the $\gamma$-translation of $\psi_\epsilon$ which is given by 
$$T_\gamma \psi_{\epsilon}(x,y,z)= \psi_{\epsilon}(\gamma^{-1}(x,y,z)).$$ 
Our goal  here is to employ the current results and study the frame property of the family $\{T_\gamma\psi_{\epsilon}\}_{\gamma\in \Gamma}$ for its spanned vector space $V_{\Gamma, \psi_\epsilon}= \overline{\text{span}\{T_\gamma \psi_\epsilon: \gamma\in \Gamma\}}$. It is obvious that $V_{\Gamma, \psi_\epsilon}$ is $\Gamma$-translation-invariant subspace of $L^2(\Bbb H^d)$.

For fixed $\epsilon$, define  $w_\epsilon$ on $\Bbb R$ by
$$w_\epsilon(\alpha) = \sum_{j \in \Z} \|  \Psi_\epsilon (\alpha + j)\|^2_{\HS(L^2(\R^d))}|\alpha + j|^d.$$
The function $w_\epsilon$ is a positive and  periodic function. Let $E_{w_\epsilon}:=\{\alpha\in (0,1): \ w_\epsilon(\alpha)>0\}$. The definition of $\Psi_\epsilon$ along  (\ref{value of HS-norm}) yields the following result.
 
\begin{lemma}\label{lem:example1} 
For  any $\alpha \in E_{w_\epsilon}$
\begin{align}\label{inequality for the toy weight}
\epsilon^d \leq w_\epsilon(\alpha) \leq 1 .
\end{align}
\end{lemma}
 
Let $k:=(0,0,k) \in \Bbb Z$ and   $T_k\psi_\epsilon$ denote the translation of $\psi_\epsilon$ at the center direction of   the Heisenberg: 
$$T_k \psi_\epsilon(p,q,t) =\psi_\epsilon (p,q,t-k), \quad (p,q,t)\in \Bbb H^d .$$

Let $\H= \overline{\text{span}\{T_k \psi_\epsilon:  \ k\in \Bbb Z\}}$ and $f\in \H$. For any  $\alpha\in (0,1)$ define 
\begin{align}\label{isometry S}
S(f)(\alpha) := 1_{E_{w_\epsilon}}(\alpha) w_\epsilon(\alpha)^{-1}   \sum_{j \in \Z} \langle \hat \psi_\epsilon (\alpha + j), \hat f(\alpha+j)\rangle_{\HS(L^2(\R^d))}|\alpha + j|^d. 
\end{align}

\begin{lemma}\label{isometry lemma}
The map $S: \H \to L_{w_\epsilon}^2(0,1)$ defined  in (\ref{isometry S})  is an unitary map. 
\end{lemma} 
\begin{proof}   
First we prove that $S$ is a bounded map on $L^2(\Bbb H^d)$. Let $f\in L^2(\Bbb H^d)$. 
Then 
\begin{align}\notag
& \int_0^1 |Sf(\alpha)|^2 w_\epsilon(\alpha)d\alpha \\\notag
& = \int_0^1  1_{E_{w_\epsilon}}(\alpha) w_\epsilon(\alpha)^{-2} \left(\sum_{j \in \Z} |\langle \hat \psi_\epsilon (\alpha + j), \hat f(\alpha+j)\rangle_{\HS(L^2(\R^d))}||\alpha + j|^d \right)^2 w_\epsilon(\alpha) d\alpha \\\notag
& \leq \int_0^1  1_{E_{w_\epsilon}}(\alpha)   w_\epsilon(\alpha)^{-1} \left(\sum_{j \in \Z} \|\hat \psi_\epsilon (\alpha + j)\|_{\HS} \|\hat f(\alpha+j)\|_{\HS} |\alpha + j|^d \right)^2 d\alpha \\\notag
&\leq \int_0^1  1_{E_{w_\epsilon}}(\alpha) w_\epsilon(\alpha)^{-1} B_{\psi_\epsilon}(\alpha) B_f(\alpha)   d\alpha\\\notag
&= \int_0^1 1_{E_{w_\epsilon}}(\alpha) B_f(\alpha)   d\alpha 
\end{align} 
where $B_g(\alpha) := \sum_j \|\hat g(\alpha+j)\|^2 |\alpha+j|^d$ for $g\in L^2(\Bbb H^d)$ and $\alpha\in (0,1)$. By the definition of $w_\epsilon$, it is immediate that $w_\epsilon(\alpha)^{-1} B_{\psi_\epsilon}(\alpha)  = 1$ for a.e. $\alpha\in (0,1)$. 
Therefore 
\begin{align}\notag
 \int_0^1 |Sf(\alpha)|^2 w_\epsilon(\alpha)d\alpha \leq
 \int_0^1 1_{E_{w_\epsilon}}(\alpha)  B_f(\alpha)   d\alpha 
 = \| f\|^2 \quad \quad \text{by (\ref{periodization})}.
\end{align}
This proves that $S$ is a bounded operator. Next we prove that $S$ is an isometry map on $\mathcal{H}$. Assume that $f= \sum_k a_k T_k \psi_\epsilon$ for any finite linear combination of $T_k\psi_\epsilon$, $k\in\Bbb Z$. Thus 
\begin{align}\notag
& \int_0^1 |Sf(\alpha)|^2 w_\epsilon(\alpha)d\alpha \\\notag
& = \int_0^1 1_{E_{w_\epsilon}}(\alpha) w_\epsilon(\alpha)^{-2} \left|\sum_{j \in \Z} \langle \hat \psi_\epsilon (\alpha + j), \hat f(\alpha+j)\rangle_{\HS(L^2(\R^d))}|\alpha + j|^d \right|^2 w_\epsilon(\alpha) d\alpha \\\notag
& =\int_0^1 1_{E_{w_\epsilon}}(\alpha) w_\epsilon(\alpha)^{-1} \left|\sum_{j \in \Z} \langle \hat \psi_\epsilon (\alpha + j), \sum_k a_k \widehat{T_k \psi_\epsilon}(\alpha+j)\rangle_{\HS(L^2(\R^d))}|\alpha + j|^d \right|^2 d\alpha \\\notag
& =\int_0^1 1_{E_{w_\epsilon}}(\alpha) w_\epsilon(\alpha)^{-1}
\left|\sum_k a_k e^{-2\pi i k \alpha}\right|^2 \left|\sum_{j \in \Z} \langle \hat \psi_\epsilon (\alpha + j), \hat\psi_\epsilon(\alpha+j)\rangle_{\HS(L^2(\R^d))}|\alpha + j|^d \right|^2 d\alpha\\\notag
& =\int_0^1  1_{E_{w_\epsilon}}(\alpha) \left|\sum_k a_k e^{-2\pi i k \alpha}\right|^2 w_\epsilon(\alpha) d\alpha \quad \text{(by the definition of $w_\epsilon$).}\\\notag
\end{align}
Notice we can write  
$$\left|\sum_k a_k e^{-2\pi i k \alpha}\right|^2 w_\epsilon(\alpha)= \sum_j \left\| \sum_k a_k \widehat{T_k \psi_\epsilon}(\alpha+j) \right\|_{\mathcal{HS}}^2 |\alpha + j|^d .$$
Applying this in above we get 
\begin{align}\notag
\int_0^1 |Sf(\alpha)|^2 w_\epsilon(\alpha)d\alpha 
&= \int_0^1  1_{E_{w_\epsilon}}(\alpha) \sum_j \left\| \sum_k a_k \widehat{T_k \psi_\epsilon}(\alpha+j)\right\|_{\mathcal{HS}}^2 |\alpha + j|^d d\alpha\\\notag
&= \int_0^1  1_{E_{w_\epsilon}}(\alpha) \sum_j \| \hat f(\alpha+j)\|_{\mathcal{HS}}^2 |\alpha + j|^d d\alpha\\\notag
&= \|f\|^2 \quad \quad \text{by (\ref{periodization})}.
\end{align}
This completes the proof of the lemma. 
\end{proof}

\begin{theorem}
For any $k\in \Bbb Z$ and $f\in \H$ define 
$$\Lambda_k(f) = \int_{E_{w_\epsilon}}  \sum_{j \in \Z} \langle \hat \psi_\epsilon (\alpha + j), \hat f(\alpha+j)\rangle_{\HS(L^2(\R^d))}|\alpha + j|^d  e^{-2\pi i \alpha k}   d\alpha .$$
Then $\Lambda_k (f) = \langle T_k \psi_{\epsilon}, f\rangle $
and $\{\Lambda_k\}_{k\in\Z}$ is a frame for $\H$ with respect to $\Bbb C$. 
\end{theorem}
\begin{proof} 
The equation $\Lambda_k (f)= \langle T_k \psi_{\epsilon}, f\rangle $ is a result of  the Parseval identity. The family $\{\Lambda_k\}_{k\in\Z}$  is a frame for $\H$ with respect to $\Bbb C$ by  Lemmas \ref{lem:example1}, \ref{isometry lemma} and Theorem \ref{th1} (b). 
\end{proof}

\section*{Acknowledgments}
The author is deeply indebted to Dr. Azita Mayeli for several fruitful discussions and generous comments. The author wishes to thank the anonymous referees for their helpful comments and suggestions that helped to improve the quality of the paper.

\end{document}